\documentclass[12pt,a4paper]{article}
\usepackage{epsfig}
\usepackage{color} 
\usepackage{latexsym}
\usepackage{amsfonts}
\usepackage{amsmath, amssymb, amsthm}
\usepackage{hyperref}
\usepackage{graphicx}
\usepackage{multicol}
\usepackage{subcaption}

\newtheorem{theorem}{Theorem}[section]
\newtheorem{proposition}[theorem]{Proposition}

\newtheorem{corollary}[theorem]{Corollary}

\newtheorem{remark}[theorem]{Remark}

\newtheorem{problem}[theorem]{Problem}

\newtheorem*{theorem1*}{Theorem}
\newtheorem*{problem1*}{Problem}
\newtheorem*{conjecture1*}{Conjecture 1}

\newcommand{\oto}{\longleftrightarrow}

\newcommand{\dpt}{\mathtt{dpt}}
\newcommand{\ang}{\mathtt{ang}}
\newcommand{\ls}{\mathtt{ls}}
\newcommand{\hl}{\mathtt{hl}}
\newcommand{\alt}{\mathtt{alt}}
\newcommand{\ax}{\mathtt{ax}}

\newcommand{\T}{\mathbf{T}}

\newcommand{\Sh}{\boldsymbol{\Sigma}}

\begin{document}

\title{\large{\textbf{
PERFECT MATCHING INDEX\\ VS.\\ CIRCULAR FLOW NUMBER OF A CUBIC GRAPH}}}

\author{
Edita M\'a\v cajov\' a and Martin \v{S}koviera\\[3mm]
Department of Computer Science\\
Faculty of Mathematics, Physics and Informatics\\
Comenius University\\
842 48 Bratislava, Slovakia\\[2mm]
{\small\tt macajova@dcs.fmph.uniba.sk}\\[-1mm]
{\small\tt skoviera@dcs.fmph.uniba.sk}}

\date{\today}

\maketitle

\begin{abstract}
The perfect matching index of a cubic graph $G$, denoted by
$\pi(G)$,  is the smallest number of perfect matchings that
cover all the edges of $G$. According to the Berge-Fulkerson
conjecture, $\pi(G)\le5$ for every bridgeless cubic graph~$G$.
The class of graphs with $\pi\ge 5$ is of particular interest
as many conjectures and open problems, including the famous
cycle double cover conjecture, can be reduced to it. Although
nontrivial examples of such graphs are very difficult to find,
a few infinite families are known, all with circular flow
number $\Phi_c(G)=5$. It has been therefore suggested
[Electron. J. Combin. 23 (2016), $\#$P3.54] that $\pi(G)\ge 5$
might imply $\Phi_c(G)\ge 5$. In this article we dispel these
hopes and present a family of cyclically $4$-edge-connected
cubic graphs of girth at least $5$ (snarks) with $\pi\ge 5$ and
$\Phi_c\le 4+\frac23$.


\end{abstract}

\section{Introduction}

Cubic graphs that cannot be covered with four perfect matchings
have recently attracted considerable attention. The reason for
this interest stems from their close relationship to several
difficult and long-standing conjectures such as the cycle
double conjecture, the Berge-Fullkerson conjecture, and others.
The current knowledge about these graphs is very limited and
examples are extremely rare.

It is well known that every bridgeless cubic graph admits a set
of perfect matchings that cover all its edges (see~\cite{Sch}).
The smallest number of perfect matchings for such a cover is
the \emph{perfect matching index} and is denoted by $\pi(G)$.
Obviously, $\pi(G)\ge 3$ for every bridgeless cubic graph $G$,
with equality attained precisely when the graph is
$3$-edge-colourable. Although no constant upper bound is known,
the Berge-Fulkerson conjecture (see \cite{S}) suggests that
this number should not exceed $5$.

\begin{figure}[h!]
\begin{center}
\includegraphics[width=.4\linewidth]{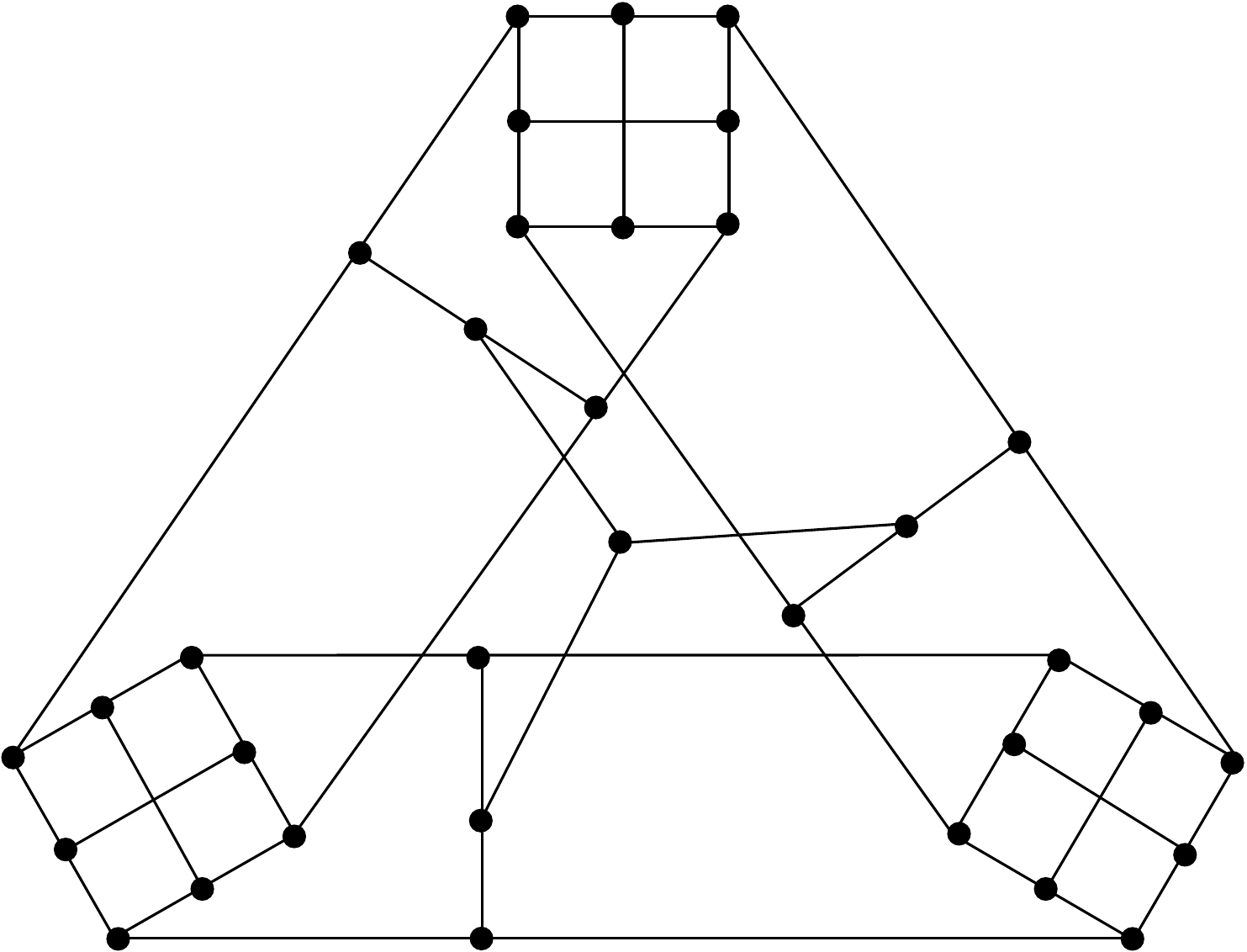}
\caption{A graph of order $34$ that cannot be covered with four perfect
matchings}\label{fig:g34}
\end{center}
\end{figure}

Nontrivial cubic graphs with perfect matching index greater
than 4 are very difficult to find. In fact, until 2013, only
one cyclically $4$-edge-connected cubic graph with $\pi\ge 5$
was known -- of course, the Petersen graph. The situation
changed after the exhaustive computer search performed by
Brinkmann et al. \cite{BGHM} revealed another such graph on 34
vertices (see Figure~\ref{fig:g34}). This graph became a
starting point for the construction of three infinite families
of graphs with this property, the windmill graphs of Esperet
and Mazzuoccolo \cite{EM}, the treelike snarks of Abreu et al.
\cite{AKLM}, and a family of Chen \cite{Chen} similar to the
windmill graphs. Esperet and Mazzuoccolo \cite{EM} also showed
that it is NP-complete to decide for a bridgeless cubic graph
$G$ whether $\pi(G)\le 4$ or $\pi(G)\ge 5$, implying that the
family of cubic graphs with $\pi\ge 5$ is sufficiently rich.

Somewhat surprisingly, all graphs with $\pi\ge 5$ known so far
have circular flow number at least 5 (see
~\cite[Theorem~9.1]{EMMS_P1}). Recall that the circular flow
number of a bridgeless graph $G$, denoted by $\Phi_c(G)$, is
the smallest rational number $r$ such that $G$ admits a
nowhere-zero $r$-flow. With similar reasons in mind, Abreu at
al.~\cite{AKLM} and Fiol at al.~\cite{FMS} suggested that cubic
graphs critical with respect to perfect matching index
(corresponding to Berge's conjecture) might be also critical
with respect to circular flow number (corresponding to Tutte's
$5$-flow conjecture). In other words, perfect matching index at
least 5 ought to imply circular flow number being at least 5.

In this paper we dispel these expectations and exhibit the
first family of cyclically $4$-edge-connected cubic graphs of
girth at least $5$ (nontrivial snarks) with $\pi\ge 5$ and
$\Phi_c<5$. In fact, we provide an infinite family of
nontrivial snarks for which  $4+\frac12 < \Phi_c\le 4+\frac23$.

Our construction heavily depends on the results of
\cite{EMMS_P1}. In that paper we have developed a theory that
describes coverings with four perfect matchings as flows whose
flow values represent points and outflow patterns represent
lines of a tetrahedron in the 3-dimensional projective space
$PG(3,2)$ over the 2-element field. The geometric
representation of coverings can be used as a powerful tool for
the study of graphs that cannot be covered with four perfect
matchings and enables a great variety of constructions of such
graphs. The main ideas of this theory are reviewed in
Section~\ref{sec:theory}, making the present article
sufficiently self-contained.

\section{Preliminaries}\label{sec:prelim}

Graphs studied in this paper will be often assembled from
smaller building blocks called multipoles. Similarly to graphs,
each \textit{multipole} $M$ has its vertex set $V(M)$, its edge
set $E(M)$, and an incidence relation between vertices and
edges. Each edge of $M$ has two ends, at least one of which is
incident with a vertex. An edge whose one end is
incident with a vertex and the other end is free is called a
\textit{dangling edge}. Free or isolated edges thus do not
occur in this paper.

A multipole with $k$ dangling edges is called a
\textit{$k$-pole}. A \textit{dipole} is a multipole whose
dangling edges are partitioned into two sets of equal size,
called \textit{connectors}. If the size is~$m$, the dipole is
an \textit{$(m,m)$-pole}.  One of the connectors of a dipole is
chosen as its \emph{input connector}; the other connector is
its \textit{output connector}. In order to avoid ambiguity,
connectors of dipoles are endowed with a fixed (but arbitrary)
linear order. All multipoles in this paper are cubic, which
means that each vertex is incident with exactly three
edge-ends.

Free ends of any two dangling edges $s$ and $t$ can be
coalesced to produce a new edge $s*t$, the \textit{junction} of
$s$ and $t$, whose end-vertices are the other end of $s$ and
the other end of $t$.

Given an $(m,m)$-pole $M_1$ and an  $(m,m)$-pole $M_2$, we can
construct a new $(m,m)$-pole $M_1\circ M_2$, the
\textit{composition} of  $M_1$ and $M_2$, by performing the
junction of the $i$-th edge of output connector of $M_1$ with
the $i$-th edge of the input connector of $M_2$. The input and
the output connectors of $M_1\circ M_2$ are inherited from
$M_1$ and $M_2$, respectively. Composition of dipoles is
clearly associative, therefore $(M_1\circ M_2)\circ
M_3=M_1\circ (M_2\circ M_3)$.

An \textit{edge-colouring} of a graph or a multipole $X$ is an
assignment of colours from a set $Z$ of \textit{colours} to the
edges of $X$ in such a way that the edges with adjacent edge
ends receive distinct colours. It means that all edge
colourings in this paper are proper. A $2$-connected cubic
graph whose edges cannot be properly coloured with three
colours is called a \textit{snark}. A snark is
\textit{nontrivial} if it is cyclically $4$-edge-connected and
has girth at least $5$.

Given an abelian group $A$, an \textit{$A$-flow} on a graph $G$
consists of an orientation of $G$ and a function $\phi\colon
E(G)\to A$ such that, at each vertex, the sum of all incoming
values equals the sum of all outgoing ones (\textit{Kirchhoff's
law}). A flow which only uses nonzero elements of the group is
said to be \textit{nowhere-zero}. An \textit{integer $k$-flow},
where $k\ge 2$ is an integer, is a $\mathbb{Z}$-flow with value
range contained in $\{0,\pm1,\ldots,\pm(k-1)\}$.

Finally, we define the \textit{total flow through a dipole $X$}
as the sum of flow-values on the dangling edges of the input
connector directed towards the dipole; of course, this value
coincides with the sum of flow-values on the dangling edges in
the output connector of $X$ directed away from~$X$.

\section{Tetrahedral flows}\label{sec:theory}

Consider a cubic graph which has a covering
$\mathcal{M}=\{P_1,P_2,P_3,P_4\}$ of its edge set with four
perfect matchings. One can clearly represent $\mathcal{M}$ by a
mapping
$$\phi\colon E(G)\to\mathbb{Z}_2^4$$
where the $i$-th coordinate of the value $\phi(e)$ equals
$1\in\mathbb{Z}_2$ whenever the edge $e$ does not belong to the
perfect matching $P_i$. It is not difficult to see that $\phi$
is a nowhere-zero $\mathbb{Z}_2^4$-flow on $G$. It may be a
little less obvious that $\phi$ has an additional geometric
structure which can be conveniently described in terms of
$3$-dimensional projective space over the $2$-element field.
More importantly, this structure proves useful. In this section
we review the main ideas of the theory and refer the reader to
our paper \cite{EMMS_P1} for details.

We start with the necessary geometric definitions. The
\textit{$n$-dimensional projective space}
$PG(n,2)=\mathbb{P}_n(\mathbb{F}_2)$ over the $2$-element field
$\mathbb{F}_2$ is an incidence geometry whose \textit{points}
can be identified with the nonzero vectors of the
$(n+1)$-dimensional vector space $\mathbb{F}_2^{n+1}$ and
\textit{lines} are formed by the triples $\{x,y,z\}$ of points
such that $x+y+z=0$. Recall that $PG(2,2)$ is the Fano plane.
Throughout this paper we will mainly encounter the
$3$-dimensional projective space $PG(3,2)$, which has $15$
points and $35$ lines.

A \textit{tetrahedron} in $PG(3,2)$ is a configuration $T$
consisting of ten points and six lines spanned by a set
$\{p_1,p_2,p_3,p_4\}$ of four points of $PG(3,2)$ in general
position; the latter means that the set constitutes a basis of
the vector space $\mathbb{F}_2^{4}$. These four points are the
\textit{corner points} of $T$. Every pair of distinct corner
points $c_1$ and $c_2$ belongs to a unique line $\{c_1, c_2,
c_1+c_2\}$ in $T$ whose third point $c_1+c_2$ is its
\textit{midpoint}. Every point $x$ of $T$ is assigned its
\textit{weight}, which equals $1$ if $x$ is a corner point and
$2$ if $x$ is a midpoint.

Any two distinct points of $T$ lie on the same line of
$PG(3,2)$ but not necessarily on a line of $T$. Those that lie
on the same line of $T$ are \textit{collinear} in~$T$,
otherwise they are \textit{non-collinear} in $T$.

\begin{figure}[htbp]
\begin{center}
     \scalebox{0.43}
     {
       \input{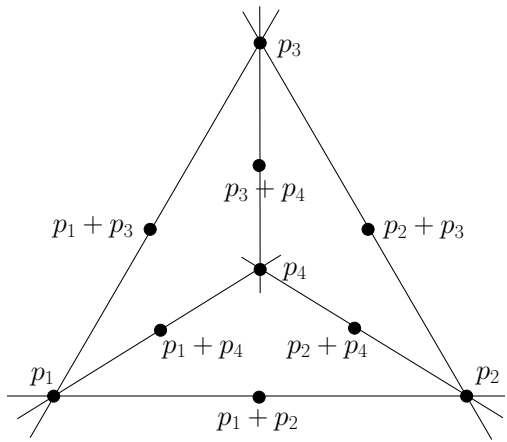}
     }
\end{center}
\caption{The tetrahedron in $PG(3,2)$ spanned by points
$p_1$, $p_2$, $p_3$, and $p_4$}\label{fig:tetra}
\end{figure}

For a given a tetrahedron $T$ in $PG(3,2)$ we define a
\textit{$T$-flow} on a cubic graph $G$ to be a mapping
$\phi\colon E(G)\to P(T)$ from the edge set of $G$ to the point
set of $T$ such that for each vertex $v$ of $G$ the three edges
$e_1$, $e_2$, and $e_3$ incident with $v$ receive values that
form a line of $T$; that is, $\phi(e_1)+\phi(e_2+\phi(e_3)=0$.
The last equation actually states that $\phi$ fulfils the
Kirchhoff law, so a $T$-flow is indeed a flow. A
\textit{tetrahedral flow} on $G$ is a $T$-flow for some
tetrahedron $T$ in $PG(3,2)$. Note that any $T$-flow is also a
proper edge-colouring.

\medskip

The following result is a cornerstone of our theory.

\begin{theorem}\label{thm:1-1}
A cubic graph can have its edges covered with four perfect
matchings if and only if it admits a tetrahedral flow.
Moreover, there exists a one-to-one correspondence between
coverings of $G$ with four perfect matchings and $T$-flows,
where $T$ is an arbitrary fixed tetrahedron in $PG(3,2)$.
\end{theorem}

A natural way of applying tetrahedral flows to the study of
cubic graphs that cannot be covered with four perfect matchings
is by analysing conflicts of tetrahedral flows on the
components resulting from the removal of an edge-cut from the
graph. If the cut-set has four edges, we can split them into
two pairs which can be regarded as the input and the output
connectors of a dipole, and inspect how pairs of points of a
tetrahedron in $PG(3,2)$ are transformed via a tetrahedral flow
from the input to the output.

Let us fix a tetrahedron $T$ with corner points $p_1$, $p_2$,
$p_3$ and $p_4$. We distinguish between six types of pairs of
points of $T$, distinct or not, which we treat as geometric
shapes.
\begin{itemize}
\item[(i)]  A \textit{line segment} is a pair $\{c_1,c_2\}$
    where $c_1$ and $c_2$ are two distinct corner points
    of~$T$. The set of all line segments of $T$ is denoted
    by $\ls$.

\item[(ii)]  A \textit{half-line} is a pair
    $\{c_1,c_1+c_2\}$  where $c_1$ and $c_2$ are two
    distinct corner points of~$T$. The set of all half-line
    of $T$ is denoted by $\hl$.

\item[(iii)] An \textit{angle} is a pair $\{c_1+c_2,
    c_1+c_3\}$ where $c_1$, $c_2$, and $c_3$ are three
    distinct corner points of $T$. The set of all angles of
    $T$ is denoted by $\ang$.

\item[(iv)] An \textit{altitude} is a pair
    $\{c_1,c_2+c_3\}$ where $c_1$, $c_2$, and $c_3$ are
    three distinct corner points of $T$. The set of all
    altitudes of $T$ is denoted by $\alt$.

\item[(v)] An \textit{axis} is a pair $\{c_1+c_2,c_3+c_4\}$
    where $c_1$, $c_2$, $c_3$, and $c_4$ are all four
    corner points of $T$ in some order. The set of all axes
    of $T$ is denoted by $\ax$.

\item[(vi)] A \textit{double point} is a degenerate pair
    $\{x,x\}$ where $x$ is any point of $T$. The set of all
    degenerate pairs of $T$ is denoted by $\dpt$.
\end{itemize}
The pairs under items (i)-(ii) are \textit{collinear}, those
under (iii)-(v) are \textit{non-collinear}. The degenerate
pairs defined in item (vi) actually occur in two varieties,
depending on whether the point $x$ is a corner point or a
midpoint, but both varieties represent the zero flow through a
connector, and from this point of view the distinction is
irrelevant.

We now define the set of \textit{shapes} to be the set
$$\Sh=\{ \ls,\hl, \ang,\alt, \ax, \dpt\}.$$
It can be shown (see \cite[Theorem~4.1]{EMMS_P1}) that for
every pair of points $\{x,y\}$ of $T$, distinct or not, there
exists a unique element $\mathtt{s}\in\Sh$ such that
$\{x,y\}\in\mathtt{s}$. This element $\mathtt{s}$ is called the
\textit{shape} of $\{x,y\}$.

The next step is to examine which pairs of shapes can occur on
the connectors of a $(2,2)$-pole equipped with a tetrahedral
flow. Consider an arbitrary (2,2)-pole $X=X(I,O)$ with input
connector $I=\{g_1,g_2\}$ and output connector $O=\{h_1,h_2\}$,
and let $T$ be a fixed tetrahedron in $PG(3,2)$. We say that
$X$ \textit{has a transition} $$\{x,y\}\to\{x',y'\}$$ or that
$\{x,y\}\to\{x',y'\}$ is a \textit{transition through} $X$, if
there exists a $T$-flow $\phi$ on $X$ such that
$\{\phi(g_1),\phi(g_2)\}=\{x,y\}$ and
$\{\phi(h_1),\phi(h_2)\}=\{x',y'\}$. If $X$ admits both
transitions $\{x,y\}\to\{x',y'\}$ and $\{x',y'\}\to\{x,y\}$, we
write
$$\{x,y\}\oto\{x',y'\}.$$
Each transition $\{x,y\}\to\{x',y'\}$ through $X$ between point
pairs induces a transition between their shapes. To be more
precise, for elements $\mathtt{s}$ and $\mathtt{t}$ of $\Sh$ we
say that $X$ \textit{has a transition}
$$\mathtt{s}\to\mathtt{t}$$
if $X$ has a transition $\{x,y\}\to\{x',y'\}$ such that
$\mathtt{s}$ is the shape of $\{x,y\}$ and $\mathtt{t}$ is the
shape of $\{x',y'\}$. The set of all transitions through $X$
reduced to their shapes forms a binary relation $\T(X)$ on
$\Sh$.

For convenience, we refer to the symbols $\{x,y\}\to\{x',y'\}$
and $\mathtt{s}\to\mathtt{t}$, with any pair of shapes, as
\textit{transitions} even without any connection to a
particular dipole and a tetrahedral flow. There is no danger of
confusion with transitions \textit{through} a dipole defined
above, which require the existence of a certain flow through
it.

Similarly to dipoles, we can also compose their transition
relations. As expected, transitions $\mathtt{p}\to\mathtt{s}$
and $\mathtt{s}\to\mathtt{t}$ of $(2,2)$-poles $X_1$ and $X_2$,
respectively, give rise to the transition
$\mathtt{p}\to\mathtt{t}$ of $X_1\circ X_2$. Conversely, a
transition $\mathtt{p}\to\mathtt{q}$ through $X_1\circ X_2$
occurs only when there exist transitions
$\mathtt{p}\to\mathtt{s}$ through $X_1$ and
$\mathtt{s}\to\mathtt{t}$ through $X_2$ for a suitable shape
$\mathtt{s}\in\Sh$. These definitions immediately imply that
$\T(X_1\circ X_2)=\T(X_1)\circ\T(X_2)$.

\medskip

The following theorem proved in \cite[Theorem~5.1]{EMMS_P1} is
essentially a consequence of Kirchhoff's law.

\begin{theorem}\label{thm:admiss-trans}
All transitions through an arbitrary $(2,2)$-pole $X$ have the
form $\mathtt{s}\to\mathtt{s}$ for some $\mathtt{s}\in\Sh$
except possibly the transitions $\ls\to\ang$ or $\ang\to\ls$.
\end{theorem}

The previous theorem implies that the transition relation
$\T(X)$ of every $(2,2,)$-pole $X$ is contained in the set
\begin{align}\label{eq:A}
\mathcal{A}=\{ &\dpt\to\dpt, \hl\to\hl, \alt\to\alt, \ax\to\ax, \nonumber \\
               &\ang\to\ang, \ang\to\ls, \ls\to\ang, \ls\to\ls \}.
\end{align}
The elements of $\mathcal{A}$ will be called \textit{admissible
transitions}.

Two types of dipoles are of particular interest. A
\textit{decollineator} is a $(2,2)$-pole with no transition
$\{x,y\}\to \{x',y'\}$ such that both $\{x,y\}$ and $\{x',y'\}$
are collinear. Among the admissible transitions only those of
type $\ls\to\ls$ and $\hl\to\hl$ are collinear, therefore every
decollineator $D$ has its transition relation $\T(D)$ contained
in the set
\begin{align}\label{eq:D}
 \mathcal{D}=\{\dpt\to\dpt, \alt\to\alt, \ax\to\ax,
               \ang\to\ang, \ang\to\ls, \ls\to\ang\}.
\end{align}

A \textit{deangulator} is a $(2,2)$-pole with no transition of
the form $\ang\to\ang$. Decollineators and deangulators are
closely related: if $D_1$ and $D_2$ are decollineators and
$U_1$ and $U_2$ are deangulators, then $D_1\circ U_i\circ D_2$
is a decollineator and $U_1\circ D_i\circ U_2$ is a deangulator
for each $i\in\{1,2\}$, see \cite[Proposition~7.3]{EMMS_P1}.

The next theorem (see \cite[Theorem~5.4]{EMMS_P1}) explains the
relationship between decollineators and cubic graphs with
perfect matching index at least $5$.

\begin{theorem}\label{thm:char-decol}
The following two statements are equivalent for an arbitrary
$(2,2)$-pole~$X$.
\begin{itemize}
\item[{\rm (i)}] $X$ is a decollineator, that is, $X$
    admits no collinear transition.

\item[{\rm (ii)}] The cubic graph $G$ created from $X$ by
    adding to $X$ two adjacent vertices and attaching each
    of them to a connector of $X$ has $\pi(G)\ge5$.
\end{itemize}
\end{theorem}

\section{A new family of graphs with $\pi\ge 5$}
\label{sec:construction}

In this section we present a new family of cubic graphs with
perfect matching index at least $5$. The reasons for high
perfect matching index of its members are quite different from
those found in the previously known families, the windmill
graphs \cite{EM}, the treelike snarks \cite{AKLM}, the snarks
of Chen \cite{Chen}, and in their common generalisation, the
Halin snarks, introduced in \cite{EMMS_P1}. While all Halin
snarks have circular flow number at least $5$ (see
\cite[Theorem~9.1]{EMMS_P1}), the new family contains an
infinite subfamily whose members have circular flow number at
most $4+\frac23$. The latter property will be established in
the next section.

Before proceeding to the construction we need a few definitions.

According to Theorem~\ref{thm:admiss-trans}, every $(2,2)$-pole
$X$ has $\T(X)\subseteq\mathcal{A}$, where $\mathcal{A}$ is the
set of admissible transitions defined by $(\ref{eq:A})$. Let
$\mathcal{L}$ be any subset of $\mathcal{A}$. A $(2,2)$-pole
$X$ will be called an \textit{$\mathcal{L}$-dipole} if
$\T(X)\subseteq\mathcal{L}$. For example, an
$\mathcal{L}$-dipole with $\mathcal{L}=\mathcal{A}-\{\hl\to\hl,
\ls\to\ls\}=\mathcal{D}$ is a decollineator and one where
$\mathcal{L}=\mathcal{A}-\{\ang\to\ang\}$ is a deangulator.

\medskip

The next two propositions prepare the building blocks for our
construction. As we shall see, they are decollineators of the
form $D_1\circ X \circ D_2$, where $D_1$ and $D_2$ are
arbitrary decollineators and $X$ is a deangulator with a
special transition relation.

\begin{proposition}\label{prop:deangQ}
Let $G$ be a cubic graph with $\pi(G)\ge 5$ containing two
$5$-cycles $C_1$ and $C_2$ such that $C_1\cap C_2$ is a path of
length $2$. Let $e$ and $f$ be the edges of $C_1\cup C_2$ that
are not incident with any vertex of $C_1\cap C_2$. Let $X$ be a
$(2,2)$-pole constructed from $G$ by severing $e$ and $f$ and
forming the connectors from the half-edges of the same edge.
Then each transition through $X$ has the form
$$\hl\to\hl, \quad \ls\to\ls, \quad \alt\to\alt,
\quad \text{and} \quad \ang\oto\ls.$$
\end{proposition}

\begin{proof}
Let $\{x,y\}\to\{x',y'\}$ be an arbitrary transition through
$X$, and let $\phi$ be a tetrahedral flow on $X$ that induces
it. First observe that $x\ne y$ for otherwise the Kirchhoff law
would imply that $x'=y'$ and hence $\phi$ would induce a
tetrahedral flow on $G$; this is impossible by
Theorem~\ref{thm:1-1}. Hence $X$ has no transition of the form
$\dpt\to\dpt$.  Next we show $X$ has no transition
$\{x,y\}\to\{x',y'\}$ where $|x|=|y|=|x'|=|y'|=2$.  Indeed, if
it had, then both edges contained in $C_1\cap C_2$ would be
forced to receive values of weight $2$ in spite of the fact
that they are adjacent (see Figure~\ref{fig:Q}). This excludes
from $\T(X)$ all admissible transitions involving an axis or an
angle except $\ang\oto\ls$. What remains are exactly those
transitions that are mentioned in the statement.
\end{proof}

\begin{figure}[htbp]
\begin{center}
     \scalebox{0.45}
     {
       \input{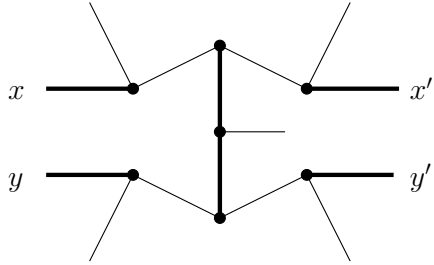}
     }
\end{center}
\caption{Excluding certain transitions in the proof of
Proposition~\ref{prop:deangQ}. Edges carrying a value of weight
$2$ are represented by bold lines.}\label{fig:Q}
\end{figure}

For the set of transitions mentioned in the statement of
Proposition~\ref{prop:deangQ} we put
\begin{align}\label{eq:Q}
\mathcal{Q}=\{\hl\to\hl, \ls\to\ls, \alt\to\alt, \ang\oto\ls \}.
\end{align}
Note that every $\mathcal{Q}$-dipole is a deangulator.

The crucial role in our construction is played by heavy
$(2,2)$-poles. We say that an edge $e$ of a multipole is
\emph{heavy} with respect to a tetrahedral flow $\phi$ if the
weight of $\phi(e)$ equals 2. A $(2,2)$-pole $X$ is
\emph{heavy} if it has at least two heavy dangling edges for
each tetrahedral flow.

\medskip

The following proposition offers a recipe for constructing
heavy $(2,2)$-poles.

\begin{proposition}\label{prop:DQD}
Let $D_1$ and $D_2$ be decollineators and let $Q$ be a
$\mathcal{Q}$-dipole, where $\mathcal{Q}$ is defined by
$(\ref{eq:Q})$. Then every transition through $D_1\circ Q\circ
D_2$ has the form
$$\ls\oto\ang, \quad \ang\to\ang,
\quad \text{and} \quad \alt\to\alt.$$
In particular, $D_1\circ Q\circ D_2$ is a heavy dipole.
\end{proposition}

\begin{proof}
Since $D_1$ and $D_2$ are decollineators, we  have
$\T(D_1)\subseteq\mathcal{D}$ and
$\T(D_2)\subseteq\mathcal{D}$. It follows that $\T(D_1\circ
Q\circ D_2) \subseteq
\mathcal{D}\circ\mathcal{Q}\circ\mathcal{D}$, which leaves the
transitions listed in the statement. It is straightforward to
check that the remaining transitions indeed make $D_1\circ
Q\circ D_2$ a heavy dipole.
\end{proof}

For the set of transitions mentioned in the statement of
Proposition~\ref{prop:DQD} we put
\begin{align}\label{eq:R}
\mathcal{R}=\mathcal{D}\circ\mathcal{Q}\circ\mathcal{D}
=\{\ls\oto\ang,\ang\to\ang, \alt\to\alt\}.
\end{align}
Note that every $\mathcal{R}$-pole is a heavy decollineator.

\begin{remark}\label{rem:PetersenQ} {\rm
Let $D_1=D_2=D_{Ps}$ where $D_{Ps}$ is the decollineator
obtained from the Petersen graph by removing two adjacent
vertices and including two dangling edges in the same connector
whenever they were formerly incident with the same vertex. It
can be verified that $\T(D_{Ps})=\mathcal{D}$. Let $Q_{Ps}$ be
the $(2,2)$-pole arising from the Petersen graph by severing
two edges at distance $2$. It is easy to see that $Q_{Ps}$
satisfies the assumptions of Propsition~\ref{prop:deangQ}, so
$\T(Q_{Ps})\subseteq \mathcal{Q}$ where $\mathcal{Q}$ is the
transition set defined by $(\ref{eq:Q})$. It is not difficult
to verify that in fact $\T(Q_{Ps})=\mathcal{Q}$. Furthermore,
by Proposition~\ref{prop:DQD} $\T(D_{Ps}\circ Q_{Ps}\circ
D_{Ps})\subseteq\mathcal{R}$. Again, it can be checked that
$\T(D_{Ps}\circ Q\circ D_{Ps})=\mathcal{R}$. Thus $D_{Ps}\circ
Q\circ D_{Ps}$ is a heavy dipole.}
\end{remark}

Now we are ready for the construction of a new family of cubic
graphs with $\pi\ge 5$.

\medskip \noindent \textbf{Construction.} Let $G$ be a cubic
graph. A new graph $\tilde G$ is constructed as follows.
\begin{itemize}
\item Replace every vertex $v$ of $G$ with a pair of
    independent vertices $v_1$ and $v_2$ in such a way that
    $\{u_1,u_2\}\cap\{w_1,w_2\}=\emptyset$ whenever $u\ne
    w$. Any vertex $v_i$ of $\tilde{G}$,
    where $v$ is a vertex of $G$ and $i\in\{1,2\}$, is called a
    \textit{lift} of $v$.
\item Replace each edge $e$ of $G$ with a heavy $(2,2)$-pole
    $X_e=X_e(I,O)$ in such a way that
    for any two distinct edges $f$ and $h$ the dipoles
    $X_f$ and $X_h$ are disjoint. The dipoles $X_e$ are
    called \textit{superedges}.
\item For each edge $e=uv$ attach the dangling edges of the
    input connector of $X_e$ to distinct vertices in
    $\{u_1,u_2\}$ and those in the output connector to
    distinct vertices in $\{v_1,v_2\}$.
\end{itemize}
The construction of the graph $\tilde G$ can be regarded as a
special form of superposition. We call $\tilde G$ a
\textit{heavy superposition} of $G$. A heavy superposition is
said to be \textit{basic} if each heavy dipole $X_e$ used for
the construction is isomorphic to the dipole $D_{Ps}\circ
Q_{Ps}\circ D_{Ps}$ from Remark~\ref{rem:PetersenQ}.

\medskip

The following theorem is the main result of this section.

\begin{theorem}\label{thm:R-superpos}
Let $G$ be a cubic graph and let $\tilde G$ be a
heavy superposition of $G$. Then $\pi(\tilde G)\ge 5$.
\end{theorem}

\begin{proof}
Assume that $G$ is a cubic graph with $n$ vertices and $m$
edges; clearly, $m=3n/2$. Suppose to the contrary that
$\pi(\tilde G)\le 4$. By Theorem~\ref{thm:1-1}, $\tilde G$
admits a tetrahedral flow, say  $\phi$. Let $W$ denote the set
of all edges of $\tilde{G}$ that are incident with some lift of
a vertex of $G$. Since each lift $v_i$ is incident with exactly
one heavy edge with respect to $\phi$, there are exactly $2n$
heavy edges in $W$. On the other hand, the set $W$ can be
decomposed into $m$ sets $W_e$ according to which superedge
$X_e$ they belong to as their dangling edges. By counting the
the average number of heavy dangling edges per superedge we
obtain
$$\frac{2n}{m}=\frac43<2.$$
This inequality implies that there exists a superedge $X_e$
with fewer than two heavy dangling edges, contradicting the
assumption that all superedges used for the construction of
$\tilde G$ are heavy.
\end{proof}

\section{Circular flows vs. perfect matching index}
\label{sec:flows}

In the preceding section have proved that a heavy superposition
$\tilde G$ of any cubic graph $G$ has perfect matching index at
least $5$. We now show that if the superposition is basic and
$G$ is $3$-edge-colourable, then the circular flow number of
$\tilde G$ is smaller than $5$.

We continue with the pertinent definitions. Given a real number
$r\ge 2$, we define a \textit{nowhere-zero real-valued
$r$-flow} as an $\mathbb{R}$-flow $\phi$ such that $1\le
|\phi(e)|\le r-1$   for each edge $e$ of $G$. A
\textit{nowhere-zero modular $r$-flow} is an
$\mathbb{R}/r\mathbb{Z}$-flow $\phi$ such that $1\le \phi(e)
\pmod r \le r-1$ for each edge $e$. The symbol $x \pmod r$
denotes the unique real number $x'\in[0,r)\subseteq\mathbb{R}$
such that $x-x'$ is a multiple of $r$. It is well known that a
graph admits a nowhere-zero real-valued $r$-flow if and only if
it admits a nowhere-zero modular $r$-flow.

The \textit{circular flow number} of a graph $G$, denoted by
$\Phi_c(G)$, is the infimum of the set of all real numbers $r$
such that $G$ has a nowhere-zero $r$-flow. It is known
\cite{Goddyn} that the \textit{circular flow number} of a
finite graph is in fact a minimum and a rational number.

\medskip

Here is the main result of this section.

\begin{theorem} \label{thm:cfn}
If $\tilde G$ is a basic heavy superposition of a
$3$-edge-colourable cubic graph $G$, then  $\tilde G$ is a
nontrivial snark  and
$$4+ \frac{1}{2}<\Phi_c(\tilde G)\le 4+ \frac{2}{3}.$$
\end{theorem}

\begin{proof}
To prove the lower bound suppose, to the contrary, that
$\Phi_c(\tilde G)\le 4+ \frac{1}{2}$. It means that $\tilde G$
has a nowhere-zero $(4+ \frac{1}{2})$-flow. In fact, by
Theorem~1.1 of \cite{Steffen-cfn}, $\tilde G$ has a
nowhere-zero $(4+ \frac{1}{2})$-flow such that every flow value
is a rational number of the form $n/2$ for some integer~$n$.
Let $\phi$ be such a flow. Clearly, $\phi$ can be taken to be a
modular $(4+\frac{1}{2})$-flow. Fix a vertex $v$ in $G$ and let
$e_1$, $e_2$, and $e_3$ be the three edges from $G$ incident
with $v$. Consider the total flow through the dipole $X_{e_i}$
in the direction from the set $\{v_1,v_2\}$, that is, the sum
of flow-values in $\mathbb{R}/(4+\frac12)\mathbb{Z}$ on the
dangling edges incident with $\{v_1,v_2\}$ directed towards the
dipole; let $h_i$ be the value. As the circular flow number of
the Petersen graph is 5, the total flow through $D_{Ps}$ lies
in the interval $(-1,1)$ modulo $4+\frac12$. and the total flow
through $Q_{Ps}$ is nonzero. Since the dipoles $D_{Ps}$ and
$Q_{Ps}$ are sequentially composed in the superedge, the total
flow through the superedge belongs to $(-1,0)\cup (0,1)$.
Taking into account the fact that the flow values are nonzero
multiples of $\frac12$, we conclude that the total flow through
any superedge is either $\frac12$ or $-\frac12$. In particular,
$h_1$, $h_2$, and $h_3$ are all in $\{-\frac12,\frac12\}$. By
the Kirchhoff law, the total outflow from the vertices $v_1$
and $v_2$ is zero, so $h_1+h_2+h_3=0$, which is clearly
impossible. Therefore $\Phi_c(\tilde G)>4+\frac12$, as claimed.

In order to establish the upper bound we construct a
nowhere-zero $(4+\frac23)$-flow on~$\tilde{G}$. To this end, it
is suffcient to find an integer $12$-flow $\phi$ such that
$|\phi(e)|\ge 3$ for each edge $e$ of $\tilde G$ and then
divide all the values by $3$. Let $\{P_1,P_2,P_3\}$ be a
1-factorisation of $G$ induced by a 3-edge-colouring. If an
edge $e$ of $G$ belongs to $P_1$, we assign values to the edges
of the corresponding superedge $X_e$ according to
Figure~\ref{fig:fsA}; similarly, if an edge belongs to $P_2$ or
$P_3$, the flow values in $X_e$ will be assigned according to
Figure~\ref{fig:fsB} and Figure~\ref{fig:fsC}, respectively. It
is easy to check that the resulting valuation and orientation
constitute an integer $12$-flow with absolute value not smaller
than $3$ on each edge of $\tilde G$. This gives rise to a
nowhere-zero $(4+\frac23)$-flow on~$\tilde{G}$.

\begin{figure}[htbp]
  \centerline{
   \scalebox{0.6}{
     \input{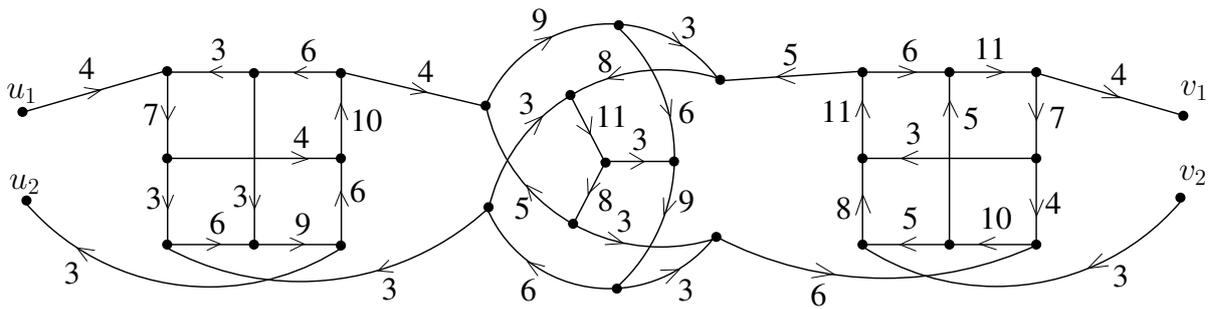}
   }
  }
   \caption{A $12$-flow on the superedges corresponding to $P_1$}
   \label{fig:fsA}
\end{figure}

\begin{figure}[htbp]
  \centerline{
   \scalebox{0.6}{
     \input{flow_supB}
   }
  }
   \caption{A $12$-flow on the superedges corresponding to $P_2$}
   \label{fig:fsB}
\end{figure}

\begin{figure}[htbp]
  \centerline{
   \scalebox{0.6}{
     \input{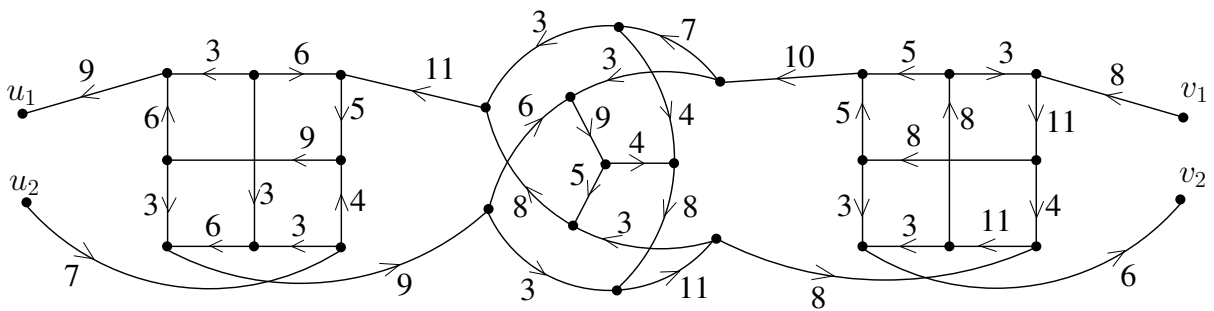}
   }
  }
   \caption{A $12$-flow on the superedges corresponding to $P_3$}
   \label{fig:fsC}
\end{figure}

At last, we show that $\tilde{G}$ is a nontrivial snark. First
note that $\tilde G$ is not $3$-edge-colourable because
$\pi(\tilde G)\ge 5$ by Theorem~\ref{thm:R-superpos}.
Furthermore, the girth of $\tilde{G}$ is obviously~5. Thus it
remains to prove that $\tilde{G}$ is cyclically
$4$-edge-connected. Take an arbitrary cycle-separating edge-cut
$S$ in $\tilde G$. It is clear that the edges of $S$ cannot all
belong to the same superedge. Therefore $S$ intersects at least
two superedges, and in each intersected superedge it has at
least two edges. Thus $|S|\ge 4$, implying that $\tilde G$ is
cyclically $4$-edge-connected. Summing up, $\tilde{G}$ is a
nontrivial snark.
\end{proof}

Note that the proof of the lower bound in Theorem~\ref{thm:cfn}
is valid for all cubic graphs $G$, not necessarily
$3$-edge-colourable ones. Furthermore, the restriction to a
basic superposition is also superfluous.

\medskip

The following statement is an immediate consequence of
Theorems~\ref{thm:R-superpos} and~\ref{thm:cfn}.

\begin{corollary}
There exist infinitely many nontrivial snarks with $\pi\ge5$
and $\Phi_c<5$.
\end{corollary}

\begin{remark}{\rm
The smallest example with $\pi\ge 5$ and $\Phi_c<5$ which
arises from our construction has 82 vertices. It is constructed
by a basic heavy superposition from the cubic graph consisting
of two vertices and three parallel edges. }
\end{remark}

We have shown, contrary to some expectations, that there exist
cubic graphs, even nontrivial snarks, with $\pi\ge 5$ and
$\Phi_c<5$. It is natural to ask how small the parameter
$\Phi_c$ can be within the family of cubic graphs that cannot
be covered with four perfect matchings. We therefore propose
the following problem.

\begin{problem}\label{prob:infimum}
What is the infimum of the set of all real numbers $r$ such
that there exists a cubic graph $G$ with $\pi(G)\ge5$ and
$\Phi_c(G)=r$?
\end{problem}

An interesting subproblem of Problem~\ref{prob:infimum} is to
determine whether there exists a constant $c>4$ such that every
cubic graph $G$ with $\pi(G)\ge 5$ has $\Phi_c(G)\ge c$.



\begin{thebibliography}{MM}\frenchspacing

\bibitem{AKLM} M. Abreu, T. Kaiser, D. Labbate, G. Mazzuoccolo,
    Treelike snarks, Electron. J. Combin. 23 (2016), $\#$P3.54

\bibitem{BGHM} G. Brinkmann, J. Goedgebeur, J. H\"agglund,
    K. Markstr\"om, Generation and properties of snarks,
    J. Combin. Theory Ser. B 103 (2013), 468--488.

\bibitem{Chen} F.~Chen, A note on Fouquet-Vanherpe's question
    and Fulkerson conjecture, Bull. Iranian Math. Soc. 42 (2016),
    1247--1258.

\bibitem{EM} L. Esperet, G. Mazzuoccolo, On cubic bridgeless
    graphs whose edge-set cannot be covered by four perfect
    matchings, J. Graph Theory 77 (2014), 144--157.

\bibitem{EMT} L. Esperet, G. Mazzuoccolo, M. Tarsi, The
    structure of graphs with circular flow number $5$ or more,
    and the complexity of their recognition problem, J. Comb. 7
    (2016), 453--479.

\bibitem{FMS} M. A. Fiol, G. Mazzuoccolo, E. Steffen, Measures
    of edge-uncolorability of cubic graphs, Electon. J. Combin.
    25 (2018),  $\#$P4.54.

\bibitem{Goddyn} L. A. Goddyn, M. Tarsi, C. Zhang, On
    $(k,d)$-colorings and fractional nowhere-zero flows, J.
    Graph Theory 28 (1998), 155--161.

\bibitem{EMMS_P1} E. M\'a\v cajov\'a, M. \v Skoviera, Cubic
    graphs that cannot be covered with four perfect matchings,
    arXiv:2008.01398 [math.CO].

\bibitem{S} P. D. Seymour, On multi-colourings of cubic
    graphs, and conjectures of Fulkerson and Tutte, Proc.
    London Math. Soc. 38 (1979), 423--460.

\bibitem{Sch} T. Sch\"onberger, Ein Beweis des Petersenschen
    Graphensatzes, Acta Litt. Sci. Szeged 7 (1934), 51--57.

\bibitem{Steffen-cfn} E. Steffen, Circular flow numbers of
    regular multigraphs, J. Graph Theory 36 (2001), 24--34.

\end{thebibliography}
\end{document}